\documentclass{amsart}
\usepackage[utf8]{inputenc}
\usepackage{amsmath}
\usepackage{amssymb}
\usepackage{caption}
\usepackage{amsthm}
\usepackage[hidelinks]{hyperref}
\usepackage{cleveref}
\usepackage{hyperref}
\usepackage{nicefrac}
\usepackage[inline]{enumitem}
\usepackage{enumitem}

\usepackage{amssymb}
\usepackage{mathrsfs}

\usepackage{array,float}

\usepackage{tikz}
\usetikzlibrary{shapes,arrows}
\usetikzlibrary{fit,positioning}
\usetikzlibrary{patterns,decorations.pathreplacing}
\usepackage{xkeyval}
\usepackage{moreverb}
\usepackage{epic}
\usepgfmodule{shapes,plot,decorations}

\usepackage{booktabs} 
\usepackage[normalem]{ulem}

\newcommand{\Q}{\mathbb{Q}}
\newcommand{\R}{\mathbb{R}}

\newcommand{\HH}{\mathbb{H}}

\newcommand{\G}{\mathbf{G}}

\usepackage[warn]{mathtext}

\newcommand{\OOO}{\mathcal{O}}

\newcommand{\id}{\mathrm{Id}}

\newcommand{\Cyc}{\mathrm{Cyc}}

\newtheorem{theorem}{Theorem}[section]
\newtheorem{corollary}{Corollary}[theorem]
\newtheorem{lemma}[theorem]{Lemma}

\newtheorem{question}[theorem]{Question}
\newtheorem{observation}[theorem]{Observation}

\theoremstyle{remark}
\newtheorem{remark}[theorem]{Remark}

\theoremstyle{remark}

\renewcommand{\qed}{\hfill$\scriptstyle\blacksquare$}

\title{Geometric and arithmetic properties of Löbell polyhedra}

\author{Nikolay Bogachev}
\address{Department of Mathematics, University of Toronto, 40 St George Street, Toronto ON, M5S~2E4, Canada}
\address{Institute for Information Transmission Problems, Moscow, Russia}
\email{nvbogach@mail.ru}

\author{Sami Douba}
\address{Institut des Hautes Études Scientifiques, 35 route de Chartres, 91440 Bures-sur-Yvette, France}
\email{douba@ihes.fr}

\begin{document}

\setcounter{tocdepth}{1}

\begin{abstract}
The L\"obell polyhedra form an infinite family of compact right-angled hyperbolic polyhedra in dimension $3$.
We observe, through both elementary and more conceptual means, that the ``systoles'' of the L\"obell polyhedra approach $0$, so that these polyhedra give rise to particularly straightforward examples of closed hyperbolic $3$-manifolds with arbitrarily small systole, and constitute an infinite family even up to commensurability. By computing number theoretic invariants of these polyhedra, we refine the latter result, and also determine precisely which of the L\"obell polyhedra are quasi-arithmetic.
\end{abstract}

\subjclass[2020]{}

\maketitle

\section{Introduction}

For each $n \geqslant 5$, consider a combinatorial $3$-polyhedron whose ``top'' and ``bottom'' faces are $n$-gons and whose ``lateral" surface consists of $2n$ pentagons (see Figure~\ref{fig:Lobell-6} and Figure~\ref{fig:Lobell}, left, for illustrations of the case $n=6$). By Andreev's theorem \cite{And70, MR2336832}, this abstract polyhedron can be realized as a compact right-angled polyhedron $L_n$ in $\mathbb{H}^3$ for any $n \geqslant 5$. For instance, the polyhedron $L_5$ is the right-angled hyperbolic dodecahedron.

The polyhedra $L_n$ are called {\em L\"obell polyhedra}, after F.\,R.~Löbell, who constructed   the first example \cite{lobell} of a closed oriented hyperbolic $3$-manifold by gluing eight copies of the polyhedron $L_6$. The subgroup  $\Gamma_n < \mathrm{Isom}(\mathbb{H}^3)$ generated by the reflections in the faces of $L_n$ is a cocompact right-angled reflection group; we refer to the quotients $\mathbb{H}^3/\Gamma_n$ as {\em L\"obell orbifolds}. Such an orbifold $\HH^3/\Gamma_n$ can be visualized as the polyhedron $L_n$ itself with reflective singularities in its faces. The example of L\"obell is a particular instance of a general construction of degree-$8$ manifold covers of right-angled reflection $3$-orbifolds; see \cite{MR924975} or \cite[Section~3.1]{MR3635439}.

Much is understood about L\"obell polyhedra. In particular, Vesnin \cite{MR1694014} computed their volumes $\mathrm{vol}(L_n)$, and showed that $\mathrm{vol}(L_n)$ is an increasing function of $n$. Moreover, Mednykh and Vesnin \cite{MR2465443} showed that the distance $\delta_n$ between the top and bottom faces of the L\"obell polyhedron $L_n$ satisfies
\begin{equation*}
\cosh \delta_n =  \frac{\cos\left( \frac{\pi}{n} \right)}{\cos \left( \frac{2\pi}{n} \right)}.
\end{equation*}
The following is an elementary consequence of the above identity.

\begin{observation}\label{ob:topandbottom}
The distance $\delta_n$ between the top and bottom faces of the L\"obell polyhedron $L_n$ approaches $0$ as $n \rightarrow \infty$.
\end{observation}

\noindent Since the systole of a compact right-angled reflection orbifold is twice the minimal distance between two nonadjacent walls (see Section \ref{sec:systoles} for the definition of the systole of a hyperbolic orbifold), and by an application of a collar lemma as in \cite{MR1273264} to the lateral faces of $L_n$, whose areas are fixed, the following is immediate from Observation \ref{ob:topandbottom}.

\begin{corollary}\label{cor:precisesystole}
    For sufficiently large $n$, the systole of the L\"obell orbifold $L_n$ is precisely $2\delta_n$.
\end{corollary}

\noindent Observation \ref{ob:topandbottom} also yields the following.

\begin{corollary}
Let $M_n$ be manifold covers of the L\"obell orbifolds $L_n$ of uniformly bounded degree (for instance, the degree-$8$ manifold covers discussed above). Then the systole of $M_n$ approaches $0$ as $n \rightarrow \infty$.
\end{corollary}

\noindent Indeed, for $n$ a multiple of $3$, we provide an explicit degree-$8$ oriented manifold cover of $L_n$ whose systole is precisely $2\delta_n$ for sufficiently large $n$; see Figure~\ref{fig:coloring}. We thus obtain straightforward examples of closed hyperbolic $3$-manifolds with arbitrarily small systole. Such manifolds were known to exist in dimension $3$ by Thurston's hyperbolic Dehn filling theorem (see, for instance, \cite[Sections~E.5~and~E.6]{MR1219310}). In fact, the L\"obell polyhedron $L_n$ decomposes into $2n$ copies of a polyhedron $T_n$ (see Section \ref{sec:construction}) which may be viewed as a Dehn filling of a cusped polyhedron $T_\infty$, so that Observation \ref{ob:topandbottom} ultimately also follows from a Dehn filling argument; see Section \ref{sec:dehn}.

\begin{figure}
    \centering
    \includegraphics[scale=0.9]{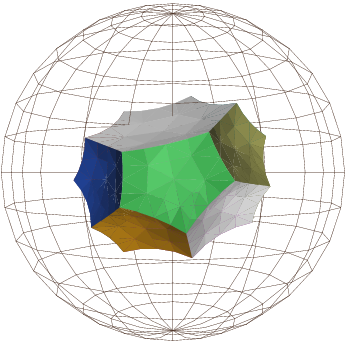}
    \caption{The Löbell polyhedron $L_6$, displayed in the Poincar\'e ball model of $\HH^3$ using {\color{blue}\underline{\href{www.geomview.org}{Geomview}}}.}
    \label{fig:Lobell-6}
\end{figure}

\begin{figure}
    \centering
    \includegraphics[scale=0.2]{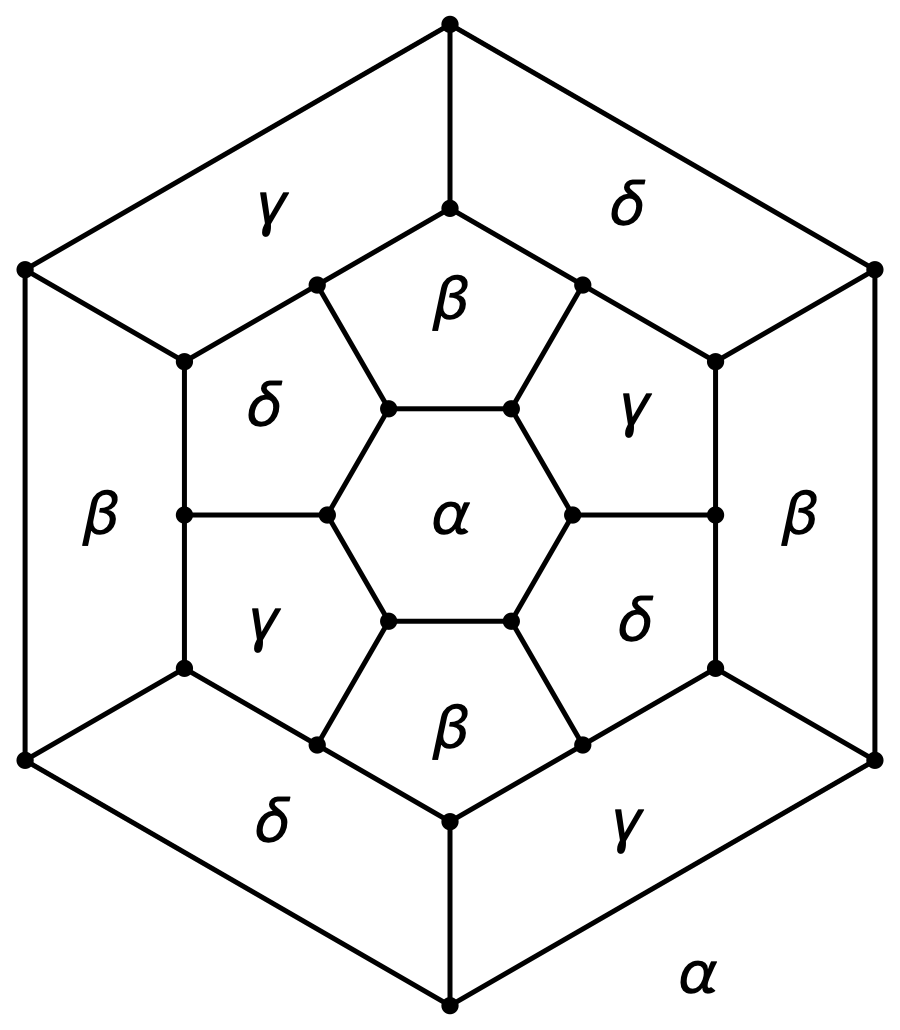}
    \caption{The above $4$-coloring of the L\"obell polyhedron $L_6$ and its analogues for the polyhedra $L_{3k}$, $k \geq 2$, determine (orientable) degree-$8$ manifold covers $M_{3k}$ as in \cite[Section~3.1]{MR3635439}. Since the top and bottom faces of $L_{3k}$ receive the same color, the manifold $M_{3k}$ contains a closed geodesic of length $2\delta_{3k}$. For sufficiently large $k$, the systole of $M_{3k}$ is in fact precisely $2\delta_{3k}$ by Corollary \ref{cor:precisesystole}.}
    \label{fig:coloring}
\end{figure}

Agol \cite{agol2006systoles} provided another construction that, given an input arithmetic lattice $\Gamma < \mathrm{Isom}(\mathbb{H}^3)$ and $\varepsilon > 0$, outputs a closed hyperbolic $3$-manifold with systole at most $\varepsilon$. Agol originally suggested this construction in dimension $4$ (where the problem had theretofore been open), but it evidently also applies in lower dimensions, and in fact applies in all dimensions by work of Belolipetsky and Thomson \cite{MR2821431} (alternatively, by a result of Bergeron, Haglund, and Wise \cite[Theorem~1.4]{MR2776645}). While the output lattices are nonarithmetic for fixed\footnote{Conjecturally, there is no dependence on $\Gamma$; see the discussion following Theorem 4.3 in \cite{agol2006systoles}, and \cite[Section~5.1]{MR2821431}.} $\Gamma$ and sufficiently small $\epsilon$, they are nevertheless all quasi-arithmetic, as observed by Thomson \cite{Tho16}. On the other hand, the reflection lattices $\Gamma_n$ are eventually not quasi-arithmetic. 

\begin{theorem}\label{th:quasi}
The L\"obell polyhedron $L_n$ is quasi-arithmetic if and only if  $n=~\!5$, $6,8,12$, and is properly quasi-arithmetic only when $n = 12$.
\end{theorem}

\noindent Recall that a lattice is said to be {\em properly quasi-arithmetic} if it is quasi-arithmetic but not arithmetic. See Section~\ref{sec:arith} for definitions. We remark that it is unclear to us whether Theorem \ref{th:quasi} exemplifies a more general phenomenon. More precisely, it appears that the following question is open.

\begin{question}\label{q:dehn}
Are only finitely many Dehn fillings of a complete finite-volume noncompact hyperbolic $3$-orbifold quasi-arithmetic?
\end{question}

\noindent The answer to Question \ref{q:dehn} is known to be affirmative if one replaces ``quasi-arithmetic'' with ``arithmetic"; see Maclachlan and Reid \cite[Cor.~11.2.2]{MR1937957}. 

Returning to our discussion on L\"obell polyhedra, Vesnin \cite{Ves91} observed that the L\"obell polyhedron $L_n$ is nonarithmetic for $n \ne 5, 6, 7, 8, 10, 12, 18$. Antol\'in-Camarena, Maloney, and Roeder \cite{AMR09} later showed that $L_n$ is arithmetic if and only if $n = 5, 6, 8$. 

Our proof of Theorem \ref{th:quasi} is straightforward and uses only classical tools from Vinberg's theory of hyperbolic reflection groups. Along the way, we compute the adjoint trace fields of the lattices $\Gamma_n$.

\begin{theorem}\label{th:commens}
The adjoint trace field $k_n$ of $\Gamma_n$ is $\mathbb{Q}\left(\cos\frac{2\pi}{n}\right)$. In particular,
if $p,q \geq~\!5$ are distinct primes, then the L\"obell polyhedra $L_p$ and $L_q$ are incommensurable. 
\end{theorem}

It is shown in \cite[Section~4.7.3]{MR1937957} that there are infinitely many pairwise incommensurable compact Coxeter polyhedra in $\mathbb{H}^3$. However, we could not find in the literature a justification of the existence of infinitely many pairwise incommensurable {\em right-angled} such polyhedra. Indeed, this was our initial motivation for considering the L\"obell polyhedra.
Since $\deg(k_n) = \frac{\phi(n)}{2} \rightarrow \infty$ as $n \rightarrow \infty$, where $\phi$ is Euler's totient function, one can in fact conclude from Theorem \ref{th:commens} that there is no infinite subsequence of $\Gamma_n$ consisting entirely of pairwise commensurable lattices. This fact is indeed already implied by Observation \ref{ob:topandbottom}; see Remark \ref{rem:systolecommensurability}. 

It is worth mentioning that the existence of infinitely many pairwise incommensurable {\em noncompact} finite-volume right-angled polyhedra in $\HH^3$ was already known. For instance, Meyer--Millichap--Trapp \cite{MR4063955} and Kellerhals \cite{kellerhals2022polyhedral} showed that the {\em ideal} right-angled antiprisms $A_n$ provide a sequence of pairwise incommensurable reflection groups with the same trace fields as those of the $\Gamma_n$. There is in fact more to be said about the relationship between these two families of right-angled polyhedra: indeed, as observed by Kolpakov \cite[Section~5.1]{MR2950475}, the L\"obell polyhedra $L_n$ can be viewed as having been obtained from the antiprisms $A_n$ via Dehn filling; see Section~\ref{sec:dehn}.

\subsection*{Acknowledgements} We are grateful to Sasha Kolpakov, Greg Kuperberg, Nicolas Tholozan, and Andrei Vesnin for helpful discussions. We thank the Institut des Hautes \'Etudes Scientifiques for hosting the first author in the fall of 2022, during which most of this work was completed. The second author was supported by the Huawei Young Talents Program.

\section{Preliminaries}

\subsection{Hyperbolic lattices}

Let $\mathbb{R}^{d,1}$ be the real vector space $\mathbb{R}^{d+1}$ equipped with the standard quadratic form $f$ of signature $(d,1)$, namely, 
$$f(x)=-x_0^2+x_1^2+\dots+x_d^2.$$

The hyperboloid $\mathcal{H}=\{x \in \mathbb{R}^{d,1} \,|\, f(x) = -1 \}$ has two connected components 
$$
\mathcal{H}^+ = \{x \in \mathcal{H} \,|\, x_0 > 0\} \text{ and } \mathcal{H}^- = \{x \in \mathcal{H} \,|\, x_0 < 0\}.
$$

The $d$-dimensional hyperbolic space $\HH^d$ is the manifold $\mathcal{H}^+$ with the Riemannian metric $\rho$ induced by restricting $f$ to each tangent space $T_p(\mathcal{H}^+)$, $p \in \mathcal{H}^+$. This hyperbolic metric $\rho$ satisfies $\cosh \rho(x, y) = - (x, y)$, where $(x, y)$ is the scalar product in $\R^{d,1}$ associated to $f$.
The hyperbolic $d$-space $\HH^d$ is known to be the unique simply connected complete Riemannian $d$-manifold with constant sectional curvature $-1$.
{\em Hyperplanes} of $\HH^d$ are intersections of linear hyperplanes of $\R^{d,1}$ with $\mathcal{H}^+$, and are totally geodesic submanifolds of codimension $1$ in $\HH^d$. 

Let $\mathrm{O}_{d,1} = \mathbf{O}(f, \R)$ be the orthogonal group of the form $f$, and $\mathrm{O}'_{d,1} < \mathrm{O}_{d,1}$ be the subgroup (of index $2$) preserving $\mathcal{H}^+$. 
The group $\mathrm{O}'_{d,1}$ preserves the metric $\rho$ on $\HH^d$, and is in fact the full group $\mathrm{Isom}(\mathbb{H}^d)$ of isometries of the latter.

If $\Gamma < \mathrm{O}'_{d,1}$ is a lattice, i.e., if $\Gamma$ is a discrete subgroup of $\mathrm{O}'_{d,1}$ with a finite-volume fundamental domain in $\HH^d$, then the quotient $M=\mathbb{H}^d/\Gamma$ is a complete finite-volume {\itshape hyperbolic orbifold}. If $\Gamma$ is torsion-free, then $M$ is a complete finite-volume Riemannian manifold, and is called a {\itshape hyperbolic manifold}.

Now set $G =\mathrm{O}
'_{d,1}$, and suppose $\G$ is an admissible (for $G$) algebraic $k$-group, i.e. $\G(\R)^o$ is isomorphic to $G^o$ and $\G^\sigma(\R)$ is a compact group for any non-identity embedding $\sigma \colon k \hookrightarrow \mathbb{R}$. Then any subgroup $\Gamma < G$ commensurable with the image in $G$ of $\G(\mathcal{O}_k)$ is an \textit{arithmetic  lattice} (in $G$) with {\em ground field} $k$. 

Since $G$ also admits non-arithmetic lattices, we discuss some weaker notions of arithmeticity for lattices in $G$. Following Vinberg \cite{MR0207853}, a lattice $\Gamma < G$ is called {\em quasi-arithmetic} with {\em ground field} $k$ if some finite-index subgroup of $\Gamma$ is contained in the image in $G$ of $\G(k)$, where $\G$ is some admissible algebraic $k$-group, and is called  \textit{properly quasi-arithmetic} if $\Gamma$ is  quasi-arithmetic, but not arithmetic on the nose. 

It is worth stressing that the  notion of quasi-arithmeticity is indeed broader than that of arithmeticity; as was mentioned in the introduction, the nonarithmetic closed hyperbolic manifolds constructed by Agol \cite{agol2006systoles} and Belolipetsky--Thomson \cite{MR2821431} exist in all dimensions and, as observed by Thomson \cite{Tho16}, are quasi-arithmetic. The first examples of properly quasi-arithmetic lattices in dimensions $3$, $4$, and $5$ were constructed by Vinberg \cite{MR0207853} via reflection groups.

\subsection{Convex polyhedra and arithmetic properties of hyperbolic reflection groups}\label{sec:arith}
 
A {\em (hyperbolic) reflection group} is a discrete subgroup of $\mathrm{O}
'_{d,1}$ generated by reflections in hyperplanes. 
The fixed hyperplanes of the reflections in a finite-covolume reflection group $\Gamma < \mathrm{O}
'_{d,1}$ divide $\HH^d$ into isometric copies of a single finite-volume convex polyhedron $P \subset \HH^d$. The polyhedron $P$ is a {\em Coxeter polyhedron}, that is, a finite-sided convex polyhedron in which the dihedral angle between any two adjacent facets is an integral submultiple of $\pi$. We say $P$ is a {\em fundamental chamber} for $\Gamma$. Conversely, given a finite-volume Coxeter polyhedron $P \subset \HH^d$, the group generated by the reflections in all the supporting hyperplanes, or {\em walls}, of $P$ is a finite-covolume reflection group $\Gamma < \mathrm{O}
'_{d,1}$ with fundamental chamber $P$. We thus frequently conflate finite-volume Coxeter polyhedra in $\HH^d$ with their corresponding lattices in $\mathrm{O}
'_{d,1}$ (or their corresponding hyperbolic orbifolds).

Let  $H_e = \{x \in \HH^d \mid (x,e)=0\}$ be a hyperplane in $\HH^d \subset \R^{d,1} $ whose linear span in $\R^{d,1}$ has normal vector $e \in \R^{d,1}$ with $(e,e)=1$, and $H_e^- = \{x \in \HH^d \mid (x,e) \le 0\}$ be the half-space associated with it. If
$$P = \bigcap_{j=1}^N H_{e_j}^- $$
is a Coxeter polyhedron in $\HH^d$, 
then the matrix $G(P) = \{g_{ij}\}^N_{i,j=1} = \{(e_i, e_j)\}^N_{i,j=1}$ is its {\em Gram matrix}. We write $K(P) = \Q\left(\{g_{ij}\}^N_{i,j=1}\right)$ and denote by $k(P)$ the field generated by all possible cyclic products of the entries of $G(P)$; we call the field $k(P)$ the {\em ground field} of $P$. For convenience, the set of all cyclic products of entries of a given matrix $A = (a_{ij})^N_{i,j=1}$, i.e., the set of all possible products of the form $a_{i_1 i_2} a_{i_2 i_3} \ldots a_{i_k i_1}$, will be denoted by $\Cyc(A)$. Thus, we have $k(P) = \Q\left(\Cyc(G(P))\right) \subset K(P)$.

The following criterion allows us to determine if a given finite-covolume hyperbolic reflection group $\Gamma$ with fundamental chamber $P$ is arithmetic, quasi-arithmetic, or neither.

\begin{theorem}[Vinberg's arithmeticity criterion \cite{MR0207853}]\label{V}
Let $\Gamma < \mathrm{O}
'_{d,1}$ be a reflection group with finite-volume fundamental chamber $P \subset \HH^d$. Then $\Gamma$ is arithmetic if and only if each of the following conditions holds:
\begin{itemize}
    \item[{\bf(V1)}] $K(P)$ is a totally real algebraic number field;
    \item[{\bf(V2)}] for any embedding $\sigma \colon K(P) \to \R$, such that 
    $\sigma\!\mid_{k(P)} \ne \id$, the matrix $G^\sigma(P)$ is positive semi-definite;
    \item[{\bf(V3)}] $\Cyc(2 \cdot G(P)) \subset \OOO_{k(P)}$,
\end{itemize}
and, in this case, the ground field of $\Gamma$ is $k(P)$. The group $\Gamma$ is quasi-arithmetic if and only if it satisfies conditions {\em\textbf{(V1)}--\textbf{(V2)}}, but not necessarily {\em\textbf{(V3)}}, and, in this case, the ground field of $\Gamma$ is again $k(P)$.
\end{theorem}

\begin{remark}\label{rem:quasi-arithmetic}
Note that $2\cos \frac{\pi}{n}$ is always an algebraic integer. Thus, if there are no dashed edges in the Coxeter--Vinberg diagram of a finite-volume Coxeter polyhedron $P$, then condition  \textbf{(V3)} above automatically holds, and there is no distinction between arithmeticity and quasi-arithmeticity for the associated reflection group $\Gamma$. In particular, a triangle group acting on $\HH^2$ is quasi-arithmetic precisely when it is arithmetic.
\end{remark}

\begin{remark}\label{rem:commensurability}
    Work of Vinberg \cite[Section 4]{Vin71} implies that the ground field of a finite-covolume hyperbolic reflection group coincides with its adjoint trace field, and is thus a commensurability invariant.
\end{remark}

\section{Geometry and arithmetic of L\"obell orbifolds}

\subsection{A decomposition of $L_n$}\label{sec:construction}

For any $n \geqslant 5$ the L\"obell polyhedron $L_n$ admits a decomposition into $2n$ isometric ``slices" $T_n$, each of which may be regarded as a twice-truncated tetrahedron; this decomposition is illustrated in Figure~\ref{fig:Lobell} for $n = 6$, where the hyperbolic triangles $ABC$ and $A'B'C'$ are exactly the
results of these truncations. 

The polyhedron $T_n$ is a Coxeter polyhedron 
whose edges are labeled in   Figure~\ref{fig:Lobell}, right.  The Coxeter--Vinberg diagram for $T_n$ is given in Figure~\ref{fig:CV-diagram-T6}. The weight $d_n$ in this diagram is equal to $\cosh \delta_n$ since the distance between the top and bottom faces of $L_n$ is the same as that for $T_n$; note that $\delta_n$ is also equal to the length of the edge $AA'$.

\begin{figure}
    \centering
    \includegraphics[scale = 0.35]{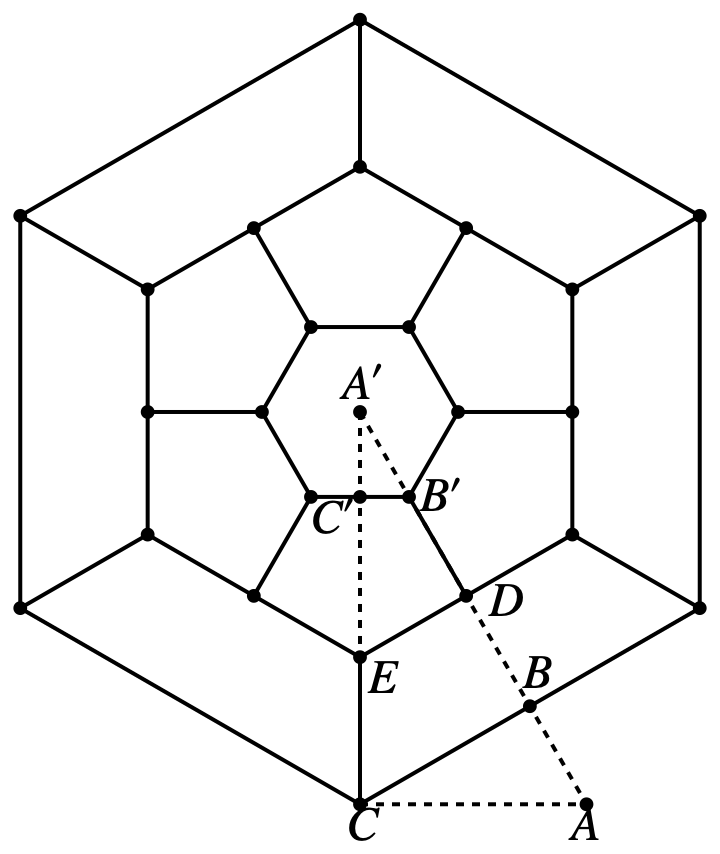} \quad
    \includegraphics[scale = 0.27]{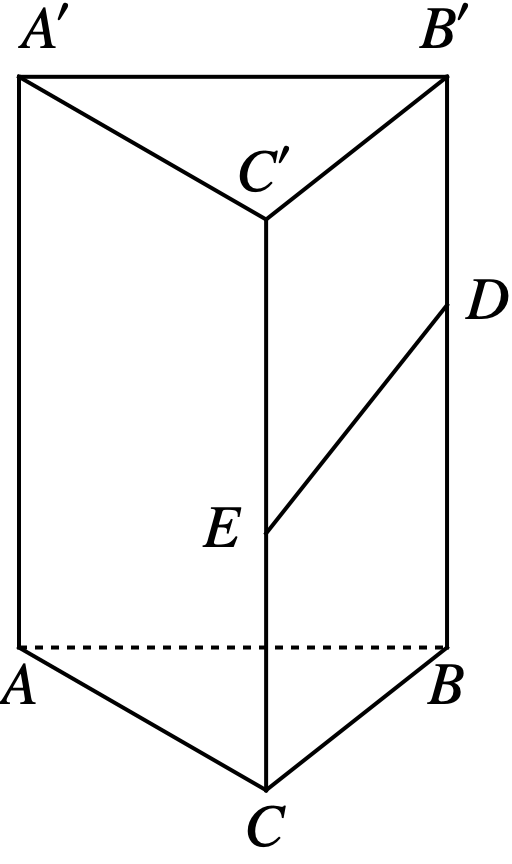}
    \caption{The Löbell polyhedron $L_6$ and its slice, the truncated tetrahedron~$T_6$. We use the analogous edge labeling for all the truncated tetrahedra $T_n$.}
    \label{fig:Lobell}
\end{figure}

The Gram matrix of $T_n$ is
$$G_n := G(T_n) = 
\newcommand{\Bold}[1]{\mathbf{#1}}\left(\begin{array}{rrrrrr}
1 & -\cos\left(\frac{\pi}{n}\right) & 0 & 0 & 0 & -\frac{1}{2} \, \sqrt{2} \\
-\cos\left(\frac{\pi}{n}\right) & 1 & 0 & 0 & -\frac{1}{2} \, \sqrt{2} & 0 \\
0 & 0 & 1 & -d_n & 0 & -a_n \\
0 & 0 & -d_n & 1 & -a_n & 0 \\
0 & -\frac{1}{2} \, \sqrt{2} & 0 & -a_n & 1 & 0 \\
-\frac{1}{2} \, \sqrt{2} & 0 & -a_n & 0 & 0 & 1
\end{array}\right)
$$

We know that the signature\footnote{The {\em signature} of a real symmetric matrix $A$ is the triple $(p,q,r)$ of numbers of positive, negative, and zero eigenvalues of $A$, respectively.} of $G_n$ is $(3,1,2)$, since $T_n \subset \HH^3$ is compact. Therefore, using the fact $\det G_n$ and all $5 \times 5$ principal minors of $G_n$ vanish, we obtain that $d_n^2 = (2a_n^2-1)(a_n^2-1)$.
This allows us to compute $d_n$ and $a_n$:
$$
d_n = \cosh \delta_n = \frac{\cos \frac{\pi}{n}}{\cos \frac{2\pi}{n}}; \quad a_n = \sqrt{1 + \frac{1}{2\cos \frac{2\pi}{n}}}.
$$

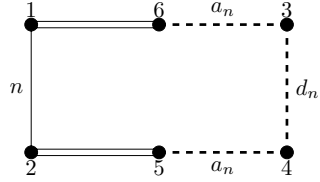
\begin{figure}
\centering
\scalebox{0.85}{
\begin{tikzpicture}
\draw[fill=black] 
 node [right] {}
(2,2) circle [radius=.1] node [above] {$6$}
(0,0) circle [radius=.1] node [below] {$2$}
(2,0) circle [radius=.1] node [below] {$5$}
(0,2) circle [radius=.1] node [above] {$1$}
(4,2) circle [radius=.1] node [above] {$3$}
(4,0) circle [radius=.1] node [below] {$4$}

(0,0) --  node [left] {$n$} (0,2)

(2,0.05) --  (0,0.05)
(2,-0.05) --  (0,-0.05)

(0,2.05) -- (2,2.05)
(0,1.95) -- (2,1.95)
;

\draw[dashed, very thick] (2,2) -- node [above] {$a_n$} (4,2);

\draw[dashed, very thick] (2,0) -- node [below] {$a_n$} (4,0);

\draw[dashed, very thick] (4,2) -- node [right] {$d_n$} (4,0);
\end{tikzpicture}}
\caption{The Coxeter--Vinberg diagram of the polytope $T_n$}\label{fig:CV-diagram-T6}
\end{figure}

\subsection{Proof of Theorem~\ref{th:commens}}

By Remark \ref{rem:commensurability}, we have that the adjoint trace field $k_n$ of the lattice $\Gamma_n$ coincides with the ground field of the polyhedron $T_n$ for $n \geq 5$. Our computations in Section~\ref{sec:construction} 
allow us to determine the cyclic products $C_n := \mathrm{Cyc}(G_n)$ using the Coxeter--Vinberg diagram shown in Figure~\ref{fig:CV-diagram-T6}.
We have
$$
\Q(C_n) = \Q\left(\left\{ \cos \frac{2\pi}{n};\quad 1 + \frac{1}{2\cos \frac{2\pi}{n}};\quad \frac{\cos^2 \frac{\pi}{n}}{\cos^2 \frac{2\pi}{n}};\quad \frac{\cos^2 \frac{\pi}{n}}{\cos \frac{2\pi}{n}}\left(1 + \frac{1}{2\cos \frac{2\pi}{n}}\right) \right\} \right).
$$
Thus
$$
k_n = \Q(C_n) = \Q\left(\cos \frac{2\pi}{n}\right).
$$
Since $\deg(k_n) = \frac{\phi(n)}{2}$, where $\phi$ is Euler's totient function, we have that for distinct primes $p$ and $q$ the fields $k_p$ and $k_q$ have different degrees, so that the polyhedra $L_p$ and $L_q$ are not commensurable. 
\qed

\subsection{Proof of Theorem~\ref{th:quasi}}
The reflection group $\Gamma_n$ is commensurable with the group $\Lambda_n$ generated by reflections in the walls of $T_n$. It thus remains to check quasi-arithmeticity of $\Lambda_n$ using Vinberg's arithmeticity criterion (see Theorem~\ref{V}).

The main result in \cite{MR4218347} implies in particular that any face of a quasi-arithmetic hyperbolic Coxeter $3$-polyhedron that is itself a Coxeter polygon is also quasi-arithmetic with the same ground field. Note that if some face $F$ of a Coxeter $3$-polyhedron meets its adjacent faces at even angles, i.e. angles of the form $\frac{\pi}{2m}$ for some $m \ge 1$, then $F$ is a Coxeter polygon. It is shown in \cite{MR4218347} that, in the latter case, if $P$ is moreover arithmetic, then $F$ is arithmetic as well.

The polyhedron $T_n$ has a $(2,4,n)$-triangular face orthogonal to all adjacent faces. For the $(2,4,n)$-triangle group, arithmeticity is equivalent to quasi-arithmeticity (see Remark~\ref{rem:quasi-arithmetic}). Takeuchi \cite{Tak77} showed that these triangle groups are  arithmetic only for  $n=5,6,7,8,10,12,18$. Thus, by the previous paragraph, we have that $\Lambda_n$ is not quasi-arithmetic for $n$ outside these values (and hence neither is $\Gamma_n$). It then suffices to check the conditions of Vinberg's criterion for the Gram matrix of the Coxeter polyhedron $T_n$ for $n$ within these values.

Denote by $\sigma_\ell$ the embeddings of the totally real number field $k_n = \Q\left(\cos \frac{2\pi}{n}\right)$, enumerated as follows:
$$
\sigma_\ell \left(\cos \frac{2\pi}{n}\right) = \cos \frac{2\pi \ell}{n}, \quad (\ell,n) = 1, \quad  1 \le \ell < n. 
$$
Note that $\sigma_\ell (k_n) = \Q\left(\cos \frac{2\pi \ell}{n}\right)$.

\begin{table}[]
    \centering
    \begin{tabular}{c|c|c|c}
        $n=5$ & $k_5 = \Q\left(\cos \frac{2\pi}{5}\right) = \Q(\sqrt{5})$ & $d_5 = 2$ & $G_5^{\sigma_\ell} \geqslant 0$ for all $\ell \ne 1$  \\
        \hline
        $n=6$ & $k_6 = \Q$ & $d_6 = 1$ & no nonidentity embeddings \\
        \hline
        $n = 7$ & $k_7 = \Q\left(\cos \frac{2\pi}{7}\right)$ & $d_7 = 3$ & $G_7^{\sigma_2}$ has signature $(3,1)$\\
        \hline
        $n = 8$ & $k_8 = \Q(\sqrt{2})$ & $d_8 = 2$ & $G_8^{\sigma_\ell} \geqslant 0$ for all $\ell \ne 1$\\
        \hline
        $n = 10$ & $k_{10} = \Q\left(\cos \frac{\pi}{5}\right) = \Q(\sqrt{5})$ & $d_{10} = 2$ & $G_{10}^{\sigma_3}$ has signature $(3,1)$ \\
        \hline
        $n = 12$ & $k_{12} = \Q\left(\cos \frac{\pi}{6}\right) = \Q(\sqrt{3})$ & $d_{12} = 2$ & $G_{12}^{\sigma_\ell} \geqslant 0$ for all $\ell \ne 1$ \\ 
        \hline
        $n = 18$ & $k_{18} = \Q\left(\cos \frac{2\pi}{18}\right)$ & $d_{18} = 3$ & $G_{18}^{\sigma_5}$ has signature $(3,1)$
    \end{tabular}
    \medskip
    \caption{Properties of the Gram matrices $G_n$ under the nonidentity embeddings $\sigma_\ell$ of the  fields $k_n$ for $n = 5,6,7,8,10,12,18$.}
    \label{tab:quasi-check}
\end{table}

We see from Table~\ref{tab:quasi-check}, which is a result of computations made in Sage, that $\Lambda_n$ is quasi-arithmetic only for $n=5,6,8,12$. In order to show that $\Lambda_n$ is properly quasi-arithmetic if and only if $n=12$, one needs to check the condition \textbf{(V3)} of Theorem~\ref{V}, that is, that $\Cyc(2 G_n) \subset \OOO_{k_n}$ if and only if $n=5, 6, 8$. This can easily be done even without a computer. To avoid being redundant, we record here only the (most interesting) case $n=12$. Indeed, the other cases were already verified in \cite{AMR09}.

\begin{lemma}
The reflection group $\Lambda_{12}$ is properly quasi-arithmetic.
\end{lemma}
\begin{proof}
Notice that $\cos(\frac{\pi}{12}) = \frac{\sqrt{6}+\sqrt{2}}{4}$. The Gram matrix $G_{12}$ of $T_{12}$ is
$$
\newcommand{\Bold}[1]{\mathbf{#1}}\left(\begin{array}{rrrrrr}
1 & -\frac{\sqrt{2}(1+\sqrt{3})}{4} & 0 & 0 & 0 & -\frac{1}{2} \, \sqrt{2} \\
-\frac{\sqrt{2}(1+\sqrt{3})}{4} & 1 & 0 & 0 & -\frac{1}{2} \, \sqrt{2} & 0 \\
0 & 0 & 1 & -\frac{\sqrt{2}(3+\sqrt{3})}{6} & 0 & -\frac{\sqrt{\sqrt{3} + 3}}{\sqrt{3}} \\
0 & 0 & -\frac{\sqrt{2}(3+\sqrt{3})}{6} & 1 & -\frac{\sqrt{\sqrt{3} + 3}}{\sqrt{3}} & 0 \\
0 & -\frac{1}{2} \, \sqrt{2} & 0 & -\frac{\sqrt{\sqrt{3} + 3}}{\sqrt{3}} & 1 & 0 \\
-\frac{1}{2} \, \sqrt{2} & 0 & -\frac{\sqrt{\sqrt{3} + 3}}{\sqrt{3}} & 0 & 0 & 1
\end{array}\right)
$$

We have that $k_{12} = \Q(\sqrt{3})$ and $$K_{12} := K(T_{12}) = \Q\left(\sqrt{2}, \sqrt{3}, \sqrt{3+\sqrt{3}}\right) = \Q\left(\sqrt{2}, \sqrt{3+\sqrt{3}}\right).$$ The latter field is of degree $8$ with all embeddings given by
$$\sqrt{2} \to \pm \sqrt{2}, \quad \sqrt{3+\sqrt{3}} \to \pm \sqrt{3 \pm \sqrt{3}}.$$ 
We verified via Sage that for each\footnote{In fact, since the semi-definiteness property can be checked via principal minors (and since these minors are computed using cyclic products contained in $k_{12} = \Q(\sqrt{3})$), it suffices to consider only a single embedding $\sigma \colon K_{12} \to \R$ mapping $\sqrt{3}$ to $-\sqrt{3}$.} embedding $\sigma: K_{12} \to \R$ such that $\sigma \mid_{k_{12}} \ne 1$, the matrix $G_{12}^\sigma$ has signature $(4,0,2)$.

Finally, we observe that not all cyclic products of $2 \cdot G_{12}$ are algebraic integers. For instance, we have that $\left(2\frac{\sqrt{\sqrt{3} + 3}}{\sqrt{3}}\right)^2 = \frac{4(3+\sqrt{3})}{3} \not\in \mathbb{Z}[\sqrt{3}] = \OOO_{k_{12}}$. 
\end{proof}

\subsection{Systoles of L\"obell orbifolds}\label{sec:systoles}

For a lattice $\Gamma < \mathrm{Isom}(\mathbb{H}^d)$, the {\em systole} $\mathrm{sys}(\Gamma)$ of $\Gamma$ is the minimal translation length of a loxodromic element of $\Gamma$. The {\em systole} $\mathrm{sys}(M)$ of the complete finite-volume hyperbolic orbifold $M = \mathbb{H}^d / \Gamma$ is simply the systole of the lattice $\Gamma$. We record in this section a couple of remarks about systoles of hyperbolic reflection orbifolds.

\begin{remark}\label{nonarithmetic}
Suppose we have a sequence of reflection groups $\Gamma_n < \mathrm{Isom}(\mathbb{H}^d)$, $d \geq 2$, with the property that $\mathrm{sys}(\Gamma_n) \rightarrow 0$. Then $\Gamma_n$ is arithmetic for at most finitely many $n$. Indeed, suppose otherwise, so that we may assume the $\Gamma_n$ are all arithmetic. Since there are only finitely many maximal arithmetic reflection groups (see \cite{MR2477273}, \cite{MR2442945}, and \cite{fisher2022new}), up to further extraction, we may also assume the $\Gamma_n$ are all contained in a single lattice $\Gamma < \mathrm{Isom}(\mathbb{H}^d)$, so that $\mathrm{sys}(\Gamma_n) \geq \mathrm{sys}(\Gamma)$, a contradiction. 
\end{remark}

\begin{remark}\label{rem:systolecommensurability}
It follows from Remark \ref{nonarithmetic} that if $\Gamma_n < \mathrm{Isom}(\mathbb{H}^d)$, $d \geq 2$, is a sequence of finite-covolume reflection groups satisfying $\mathrm{sys}(\Gamma_n) \rightarrow 0$, then for each $m \in \mathbb{N}$, we have that $\Gamma_m$ is commensurable to $\Gamma_n$ for at most finitely many $n$. Indeed, suppose otherwise. Then, up to passing to a subsequence, we may assume that the $\Gamma_n$ are all commensurable. Since the $\Gamma_n$ are nonarithmetic by Remark \ref{nonarithmetic}, it follows from a result of Margulis \cite[Theorem~1,~page 2]{MR1090825} that their commensurator $\Lambda < \mathrm{Isom}(\mathbb{H}^d)$ contains each $\Gamma_n$ as a finite-index subgroup, so that $\mathrm{sys}(\Gamma_n) \geq \mathrm{sys}(\Lambda)$, a contradiction.

As observed in \cite[Sections~5.2~and~5.3]{MR2821431}, the above conclusion in fact holds for any sequence of lattices $\Gamma_n < \mathrm{Isom}(\mathbb{H}^d)$ satisfying $\mathrm{sys}(\Gamma_n) \rightarrow 0$. Indeed, suppose one has such a sequence $\Gamma_n$ where the $\Gamma_n$ are all commensurable. If the $\Gamma_n$ are nonarithmetic, then one obtains a contradiction as in the previous paragraph. If the $\Gamma_n$ are arithmetic, then since they are commensurable, a uniform lower bound for $\mathrm{sys}(\Gamma_n)$ is provided by Lehmer's conjecture for integer polynomials of bounded degree \cite{MR296021} (see the discussion immediately following Conjecture 10.2 in Gelander \cite{MR2084613}). We remark that the strongest form of Lehmer's conjecture would imply a uniform lower bound on the systole of any arithmetic locally symmetric orbifold.
\end{remark}

\section{L\"obell orbifolds and hyperbolic Dehn fillings of ideal right-angled antiprisms}\label{sec:dehn}

Let $P \subset \mathbb{H}^3$ be a finite-volume Coxeter polyhedron with a compact edge $e$, and say the dihedral angle at $e$ is $\frac{\pi}{m}$. Following Kolpakov \cite{MR2950475}, if $Q \subset \mathbb{H}^3$ is a finite-volume Coxeter polyhedron of the same combinatorial type and with the same dihedral angles as $P$ except that the dihedral angle at $e$ in $Q$ is diminished to $\frac{\pi}{n}$ for some $n \geq m$, we say that $Q$ is obtained from $P$ via a {\em $\frac{\pi}{n}$-contraction} at $e$. For example, for $n \geq 5$ and $k \geq 2$, the polyhedron $P_{n,k} \subset \mathbb{H}^3$ whose Coxeter--Vinberg diagram is shown in Figure \ref{fig:CV-diagram-2} is obtained from the truncated tetrahedron $T_n$ via $\frac{\pi}{2k}$-contractions at the (analogues of the) edges $B'D$ and $EC$, while $T_n$ is in turn obtained from $T_5$ via a $\frac{\pi}{n}$-contraction at the edge $AA'$; see Figures \ref{fig:Lobell} and \ref{fig:CV-diagram-T6}.

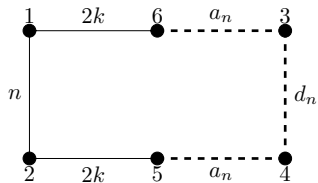
\begin{figure}[ht]
\centering
\scalebox{0.85}{\begin{tikzpicture}
\draw[fill=black] 
 node [right] {}
(2,2) circle [radius=.1] node [above] {$6$}
(0,0) circle [radius=.1] node [below] {$2$}
(2,0) circle [radius=.1] node [below] {$5$}
(0,2) circle [radius=.1] node [above] {$1$}
(4,2) circle [radius=.1] node [above] {$3$}
(4,0) circle [radius=.1] node [below] {$4$}

(0,0) --  node [left] {$n$} (0,2)

(2,0) -- node [below] {$2k$} (0,0)

(0,2) -- node [above] {$2k$} (2,2)

;

\draw[dashed, very thick] (2,2) -- node [above] {$a_n$} (4,2);

\draw[dashed, very thick] (2,0) -- node [below] {$a_n$} (4,0);

\draw[dashed, very thick] (4,2) -- node [right] {$d_n$} (4,0);
\end{tikzpicture}}
\caption{The Coxeter--Vinberg diagram of the polyhedron $P_{n,k}$}\label{fig:CV-diagram-2}
\end{figure}

If instead $Q$ and $P$ differ (as labeled polyhedra) only in that, in $Q$, the edge $e$ is replaced by an ideal vertex $v \in \partial \mathbb{H}^3$, then we say that $Q$ is obtained from $P$ by {\em contracting $e$ to an ideal vertex}. 
Note that, if such a contraction exists, the dihedral angle at each edge adjacent to $e$ in $P$ (and each edge incident to $v$ in $Q$) is $\frac{\pi}{2}$, and the faces sharing $e$ in $P$ are at least $4$-sided. 

\begin{figure}
    \centering
    \includegraphics[scale = 0.3]{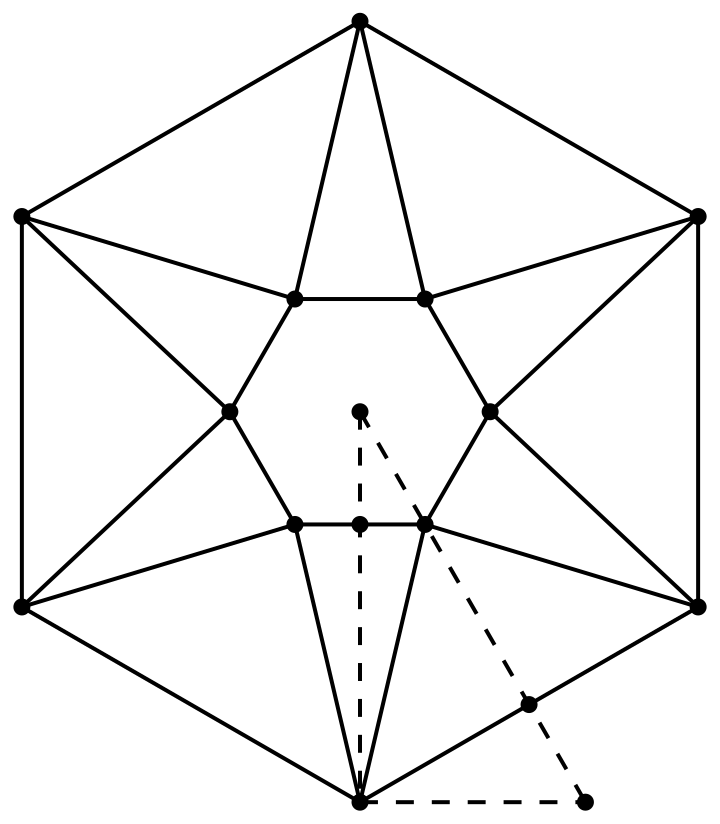} \quad
    \includegraphics[scale = 0.23]{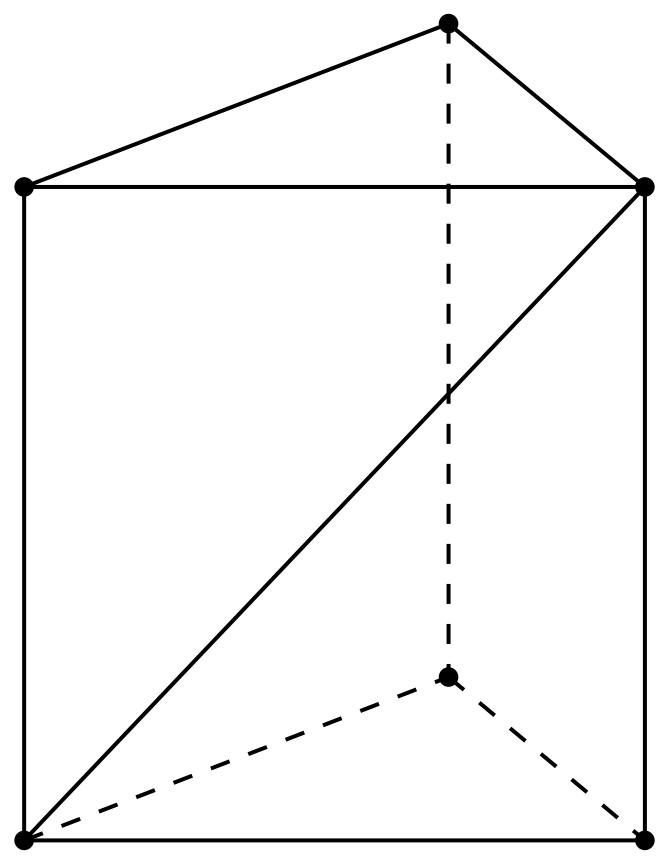}
    \caption{The antiprism $A_6$ and its slice $R_6$.}
    \label{fig:Antiprism}
\end{figure}

\begin{figure}
    \centering
    \includegraphics[scale=0.275]{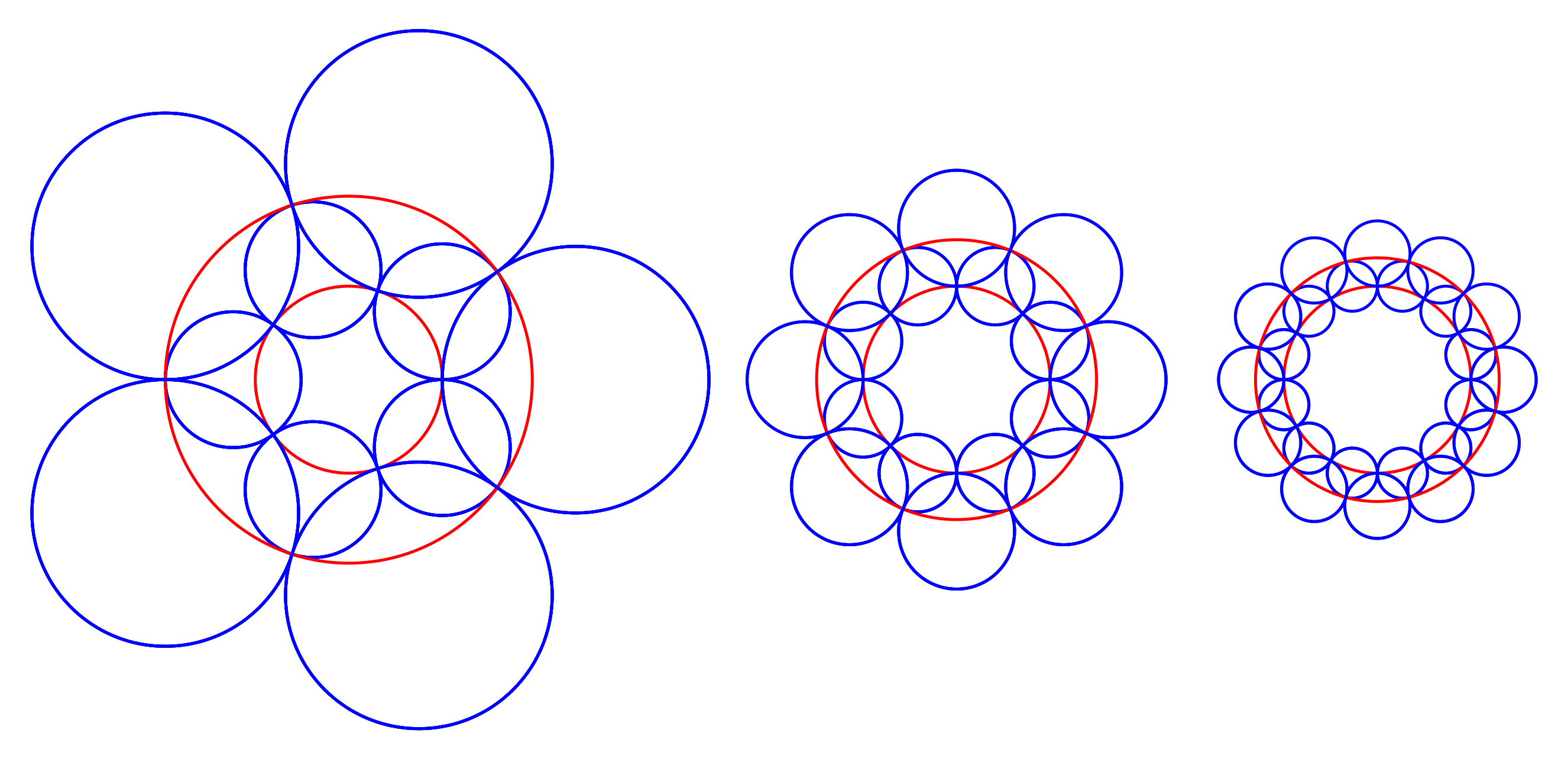}
    \caption{A visual proof that the distance between the ``top'' and ``bottom'' faces of the antiprism $A_n$ approaches $0$. Drawn above are the ideal boundaries of the walls of $A_n$ for $n=5$, $8$, and $12$ (from left to right), visualized via stereographic projection onto the page from the ideal boundary of $\HH^3$ such that the circles corresponding to the top and bottom faces, indicated here in red, are concentric. Keeping fixed the Euclidean diameter of the inner red circle, we have that as $n$ approaches $\infty$, the diameter of an inner blue circle approaches $0$, so that the ratio between the diameters of the inner and outer red circles approaches $1$. In fact, an elementary exercise in Euclidean geometry shows that this ratio is precisely $\frac{1+\sin(\frac{\pi}{n})}{\cos(\frac{\pi}{n})}$, corroborating a computation of Kellerhals \cite[Examples~1~\&~2]{kellerhals2022polyhedral}.}
    \label{fig:circles}
\end{figure}

 In \cite{kellerhals2022polyhedral}, Kellerhals studies a family of ideal right-angled polyhedra known as the antiprisms $A_n \subset \mathbb{H}^3$, $n \geq 3$, where she exploits a decomposition of each such polyhedron $A_n$  into $2n$ copies of a polyhedron $R_n$, analogous to the decomposition of the L\"obell polyhedron $L_n$ into $2n$ copies of $T_n$; see Figure \ref{fig:Antiprism}. In the above language, for $n \geq 5$, the polyhedron $R_n$ is in fact obtained from $T_n$ by contracting  the edges $B'D$ and $EC$ of $T_n$ to ideal vertices. In particular, for such $n$, the antiprism $A_n$ may be obtained from the L\"obell polyhedron $L_n$ by a sequence of edge contractions to ideal vertices. This was already observed by Kolpakov \cite[Section~5.1]{MR2950475}. Indeed, the existence of the polyhedra $P_{n,k}$ and $R_n$, $n \geq 5$, $k \geq 2$, is predicted by the following result of Kolpakov, whose proof rests on Andreev's theorem \cite{And70,MR2336832}.

\begin{theorem}[c.f.~the~proof~of~Prop.~1 ~in~\cite{MR2950475}]\label{contraction}
    Let $P \subset \HH^3$ be a finite-volume Coxeter polyhedron with at least $5$ faces and a compact edge $e$ such that the dihedral angle at each edge adjacent to $e$ is $\frac{\pi}{2}$, and the faces sharing $e$ in $P$ are at least $4$-sided. If the dihedral angle at $e$ in $P$ is $\frac{\pi}{m}$, $m \geq 2$, then $P$ admits a $\frac{\pi}{n}$-contraction $P_n$ at $e$ for each $n \geq m$, as well as a contraction $P_\infty$ of $e$ to an ideal vertex. 
\end{theorem}

Since we may instead view the $P_n$ as having been obtained from $P_\infty$ via Dehn filling\footnote{More precisely, Dunbar and Meyerhoff adapt Thurston's theory of hyperbolic Dehn fillings to the setting of {\em oriented} hyperbolic 3-orbifolds. Applying this theory to the orientation covers of the $P_n$, where the $P_n$ are viewed as reflection orbifolds, one concludes that the length of the edge $e$ in $P_n$ approaches $0$ as $n \rightarrow \infty$.}, it follows from work of Dunbar and Meyerhoff \cite{MR1291531} that the length of the (analogue of the) edge $e \subset P_n$ approaches $0$ as $n \rightarrow \infty$. Taking $P_n$ to be the truncated tetrahedron $T_n$ and $e$ the edge $AA'$, we obtain another justification for Observation \ref{ob:topandbottom}. Applying a similar argument to the $R_n$ instead of the $T_n$, one concludes that the distance between the ``top" and ``bottom" faces of the antiprism $A_n$ also approaches $0$ as $n \rightarrow \infty$; for a justification of the latter that uses only elementary Euclidean geometry, see Figure \ref{fig:circles}.

We remark that the conditions of Theorem \ref{contraction} are satisfied for any compact edge $e$ of a finite-volume right-angled polyhedron $P \subset \mathbb{H}^3$. This suggests a method of constructing a family of new finite-volume right-angled polyhedra starting from the polyhedron $P$. Namely, for $n \geq 2$, let $P_n \subset \mathbb{H}^3$ be the polyhedron obtained from $P$ by a $\frac{\pi}{n}$-contraction at $e$, let $\Lambda_n < \mathrm{Isom}(\mathbb{H}^3)$ be the associated reflection group, and let $\Delta_n$ be the stabilizer of $e$ in $\Lambda_n$. Then the finite-volume polyhedron $Q_n = \bigcup_{\gamma \in \Delta_n}\gamma P_n$ is right-angled, and is compact if and only if $P$ was. Moreover, for each $m \geq 2$, we have that $\Lambda_m$ is commensurable to $\Lambda_n$ for at most finitely many $n \geq 2$, for instance, because $\mathrm{sys}(\Lambda_n) \rightarrow 0$ as $n \rightarrow \infty$ (see Remark \ref{rem:systolecommensurability}). We conclude our discussion by observing that this strategy for producing an infinite family of pairwise incommensurable finite-volume right-angled hyperbolic polyhedra fails in dimension $4$, where the existence of such a family appears to be open. 

\begin{theorem}\label{thm:dim4}
    Let $n,m \geq 3$ be of the same parity and $d \geq 4$. Suppose $P_n \subset \HH^d$ is a finite-volume Coxeter polyhedron all of whose dihedral angles are right angles except for a single dihedral angle of $\frac{\pi}{n}$. Suppose there is also a finite-volume Coxeter polyhedron $P_m \subset \HH^d$ with the same combinatorics and dihedral angles as $P_n$ except that the exceptional dihedral angle of $P_m$ is $\frac{\pi}{m}$. Then $n=m$.
\end{theorem}

\noindent We will make use of the following lemma.

\begin{lemma}\label{lem:zdense}
    Let $P$ be a finite-volume Coxeter polyhedron in $\HH^d$, $d \geq 4$, and suppose all dihedral angles of $P$ are right angles except possibly for one dihedral angle formed by walls $H_1$ and $H_2$ of $P$. Then the group $\Gamma < \mathrm{Isom}(\HH^d)$ generated by the reflections in all walls of $P$ except $H_1$ and $H_2$ is Zariski-dense in $\mathrm{Isom}(\HH^d)$.
\end{lemma}
\begin{proof}[Proof~of~Lemma~\ref{lem:zdense}]
    Let $P' \subset \mathbb{H}^d$ be the (infinite-volume) polyhedron obtained from $P$ by forgetting the walls $H_1$ and $H_2$, and let $P_i$ be the intersection of $P'$ with the hyperplane $H_i$ of $\HH^d$ for $i =1,2$. Then the $(d-1)$-dimensional right-angled hyperbolic polyhedron $P_i$ is obtained from a finite-volume such polyhedron---namely, the polyhedron $P \cap H_i$---by forgetting a single wall---namely, the intersection  $H_1 \cap H_2$. Since $d-1 \geq 3$, it follows that the subgroup of $\mathrm{Isom}(H_i)$ generated by the reflections in the walls of $P_i$ is Zariski-dense in $\mathrm{Isom}(H_i)$; see, for instance, \cite[Section~1.7]{GPS}. The lemma follows.
\end{proof}

\begin{proof}[Proof~of~Theorem~\ref{thm:dim4}]
    For $k=n,m$, let $H_k$ and $H_k'$ be the walls of $P_k$ forming the exceptional dihedral angle of $\frac{\pi}{k}$, and let $R_k$ be the union of the images of $P_k$ under the reflection group $D_k$ generated by the reflections in $H_k$ and $H_k'$. Since $n$ and $m$ have the same parity, we may choose reflections $r_k \in D_k$ such that the $(d-1)$-dimensional polyhedra $\mathrm{Fix}(r_n) \cap R_n$ and $\mathrm{Fix}(r_m) \cap R_m$ have the same combinatorics and dihedral angles and are thus isometric by Mostow--Prasad  rigidity \cite{mostow, prasad} (see also Andreev \cite{And70,And70b}).

    Now interbreed the $R_k$ along $\mathrm{Fix}(r_k) \cap R_k$ and let $R$ be the resulting finite-volume polyhedron in $\HH^d$. Then there is an obvious dihedral group $D_{\frac{n+m}{2}}$ of combinatorial symmetries of $R$ preserving dihedral angles. By Mostow--Prasad rigidity, each of these combinatorial symmetries is a hyperbolic isometry. However, by Lemma \ref{lem:zdense}, there is (up to inverses) a unique hyperbolic isometry $\gamma_k$ rotating the surface of any given slice of $R_k$ two slices over, namely, the composition of the reflections in $H_k$ and $H_k'$, so that $\gamma_k$ has order $k$. Since $\gamma_n$ and $\gamma_m$ each generate the commutator subgroup of $D_{\frac{n+m}{2}}$, it follows that $n=m$. 
\end{proof}

\bibliography{biblio.bib}{}

\begin{thebibliography}{10}

\bibitem{agol2006systoles}
{\sc I.~Agol}, {\em Systoles of hyperbolic 4-manifolds}, arXiv preprint
  math/0612290,  (2006).

\bibitem{MR2442945}
{\sc I.~Agol, M.~Belolipetsky, P.~Storm, and K.~Whyte}, {\em Finiteness of
  arithmetic hyperbolic reflection groups}, Groups Geom. Dyn., 2 (2008),
  pp.~481--498.

\bibitem{And70}
{\sc E.~M. Andreev}, {\em Convex polyhedra in {L}oba\v{c}evski\u{\i} spaces},
  Mat. Sb. (N.S.), 81 (123) (1970), pp.~445--478.

\bibitem{And70b}
\leavevmode\vrule height 2pt depth -1.6pt width 23pt, {\em Convex polyhedra of
  finite volume in {L}oba\v{c}evski\u{\i} space}, Mat. Sb. (N.S.), 83 (125)
  (1970), pp.~256--260.

\bibitem{AMR09}
{\sc O.~Antol\'{\i}n-Camarena, G.~R. Maloney, and R.~K.~W. Roeder}, {\em
  Computing arithmetic invariants for hyperbolic reflection groups}, in Complex
  dynamics, A K Peters, Wellesley, MA, 2009, pp.~597--631.

\bibitem{MR1273264}
{\sc A.~Basmajian}, {\em Tubular neighborhoods of totally geodesic
  hypersurfaces in hyperbolic manifolds}, Invent. Math., 117 (1994),
  pp.~207--225.

\bibitem{MR2821431}
{\sc M.~V. Belolipetsky and S.~A. Thomson}, {\em Systoles of hyperbolic
  manifolds}, Algebr. Geom. Topol., 11 (2011), pp.~1455--1469.

\bibitem{MR1219310}
{\sc R.~Benedetti and C.~Petronio}, {\em Lectures on hyperbolic geometry},
  Universitext, Springer-Verlag, Berlin, 1992.

\bibitem{MR2776645}
{\sc N.~Bergeron, F.~Haglund, and D.~T. Wise}, {\em Hyperplane sections in
  arithmetic hyperbolic manifolds}, J. Lond. Math. Soc. (2), 83 (2011),
  pp.~431--448.

\bibitem{MR296021}
{\sc P.~E. Blanksby and H.~L. Montgomery}, {\em Algebraic integers near the
  unit circle}, Acta Arith., 18 (1971), pp.~355--369.

\bibitem{MR4218347}
{\sc N.~Bogachev and A.~Kolpakov}, {\em On faces of quasi-arithmetic {C}oxeter
  polytopes}, Int. Math. Res. Not. IMRN,  (2021), pp.~3078--3096.

\bibitem{MR1291531}
{\sc W.~D. Dunbar and G.~R. Meyerhoff}, {\em Volumes of hyperbolic
  {$3$}-orbifolds}, Indiana Univ. Math. J., 43 (1994), pp.~611--637.

\bibitem{fisher2022new}
{\sc D.~Fisher and S.~Hurtado}, {\em A new proof of finiteness of maximal
  arithmetic reflection groups}, arXiv preprint arXiv:2207.00258,  (2022).

\bibitem{MR2084613}
{\sc T.~Gelander}, {\em Homotopy type and volume of locally symmetric
  manifolds}, Duke Math. J., 124 (2004), pp.~459--515.

\bibitem{GPS}
{\sc M.~Gromov and I.~Piatetski-Shapiro}, {\em Nonarithmetic groups in
  {L}obachevsky spaces}, Inst. Hautes \'{E}tudes Sci. Publ. Math.,  (1988),
  pp.~93--103.

\bibitem{kellerhals2022polyhedral}
{\sc R.~Kellerhals}, {\em A polyhedral approach to the arithmetic and geometry
  of hyperbolic link complements},  (2022).

\bibitem{MR2950475}
{\sc A.~Kolpakov}, {\em Deformation of finite-volume hyperbolic {C}oxeter
  polyhedra, limiting growth rates and {P}isot numbers}, European J. Combin.,
  33 (2012), pp.~1709--1724.

\bibitem{lobell}
{\sc F.~L{\"o}bell}, {\em {Beispiele geschlossener dreidimensionaler
  Clifford--Kleinscher R\"aume negativer Kr\"ummung}}, Ber. Verh. S\"achs.
  Akad. Leipzig, 83 (1931), pp.~167--174.

\bibitem{MR1937957}
{\sc C.~Maclachlan and A.~W. Reid}, {\em The arithmetic of hyperbolic
  3-manifolds}, vol.~219 of Graduate Texts in Mathematics, Springer-Verlag, New
  York, 2003.

\bibitem{MR1090825}
{\sc G.~A. Margulis}, {\em Discrete subgroups of semisimple {L}ie groups},
  vol.~17 of Ergebnisse der Mathematik und ihrer Grenzgebiete (3) [Results in
  Mathematics and Related Areas (3)], Springer-Verlag, Berlin, 1991.

\bibitem{MR2465443}
{\sc A.~D. Mednykh and A.~Y. Vesnin}, {\em L\"{o}bell manifolds revised}, Sib.
  \`Elektron. Mat. Izv., 4 (2007), pp.~605--609.

\bibitem{MR4063955}
{\sc J.~S. Meyer, C.~Millichap, and R.~Trapp}, {\em Arithmeticity and hidden
  symmetries of fully augmented pretzel link complements}, New York J. Math.,
  26 (2020), pp.~149--183.

\bibitem{mostow}
{\sc G.~D. Mostow}, {\em Quasi-conformal mappings in {$n$}-space and the
  rigidity of hyperbolic space forms}, Inst. Hautes \'{E}tudes Sci. Publ.
  Math.,  (1968), pp.~53--104.

\bibitem{MR2477273}
{\sc V.~V. Nikulin}, {\em Finiteness of the number of arithmetic groups
  generated by reflections in {L}obachevski\u{\i} spaces}, Izv. Ross. Akad.
  Nauk Ser. Mat., 71 (2007), pp.~55--60.

\bibitem{prasad}
{\sc G.~Prasad}, {\em Strong rigidity of {${\bf Q}$}-rank {$1$} lattices},
  Invent. Math., 21 (1973), pp.~255--286.

\bibitem{MR2336832}
{\sc R.~K.~W. Roeder, J.~H. Hubbard, and W.~D. Dunbar}, {\em Andreev's theorem
  on hyperbolic polyhedra}, Ann. Inst. Fourier (Grenoble), 57 (2007),
  pp.~825--882.

\bibitem{Tak77}
{\sc K.~Takeuchi}, {\em Arithmetic triangle groups}, J. Math. Soc. Japan, 29
  (1977), pp.~91--106.

\bibitem{Tho16}
{\sc S.~Thomson}, {\em Quasi-arithmeticity of lattices in
  {$\mathrm{PO}(n,1)$}}, Geom. Dedicata, 180 (2016), pp.~85--94.

\bibitem{MR924975}
{\sc A.~Y. Vesnin}, {\em Three-dimensional hyperbolic manifolds of {L}\"{o}bell
  type}, Sibirsk. Mat. Zh., 28 (1987), pp.~50--53.

\bibitem{Ves91}
\leavevmode\vrule height 2pt depth -1.6pt width 23pt, {\em Three-dimensional
  hyperbolic manifolds with general fundamental polyhedron}, Math. Notes, 49
  (1991), p.~575–577.

\bibitem{MR1694014}
\leavevmode\vrule height 2pt depth -1.6pt width 23pt, {\em Volumes of
  {L}\"{o}bell {$3$}-manifolds}, Mat. Zametki, 64 (1998), pp.~17--23.

\bibitem{MR3635439}
\leavevmode\vrule height 2pt depth -1.6pt width 23pt, {\em Right-angled
  polytopes and three-dimensional hyperbolic manifolds}, Uspekhi Mat. Nauk, 72
  (2017), pp.~147--190.

\bibitem{MR0207853}
{\sc E.~B. Vinberg}, {\em Discrete groups generated by reflections in
  {L}oba\v{c}evski\u{\i}\ spaces}, Mat. Sb. (N.S.), 72 (114) (1967),
  pp.~471--488; correction, ibid. 73 (115) (1967), 303.

\bibitem{Vin71}
{\sc E.~B. Vinberg}, {\em Rings of definition of dense subgroups of semisimple
  linear groups}, Math. of USSR--Izvestiya, {\bf 5} (1) (1971), pp.~45--55.

\end{thebibliography}
\bibliographystyle{siam}

\end{document}